\documentclass[a4paper,11pt]{amsart}

\usepackage[utf8]{inputenc}
\usepackage[T1]{fontenc}	
\usepackage{amsmath}		
\usepackage{amssymb}		
\usepackage{enumerate}
\usepackage{amsthm}
\usepackage{mathrsfs}
\usepackage{mathtools}
\mathtoolsset{showonlyrefs,showmanualtags}

\usepackage{cite}

\usepackage{setspace}

\usepackage{bm}
\usepackage{bbm} 
\usepackage{enumitem,linegoal}
\usepackage{calc}

\usepackage{stmaryrd}

\usepackage{tikz-cd}

\usepackage{tabularx}
\newcolumntype{Y}{>{\centering\arraybackslash}X}

\usepackage{caption} 
\usepackage{booktabs}

\usepackage{chngcntr}
\usepackage{apptools}
\AtAppendix{\counterwithin{thm}{section}}

\usepackage{url}

\usepackage[hidelinks]{hyperref}

\newcommand\restr[2]{{
	\left.\kern-\nulldelimiterspace 
	#1 
	\vphantom{\big|} 
	\right|_{#2} 
}}

\newcommand\srestr[2]{{
	\left.\kern-\nulldelimiterspace 
	#1 
	\right|_{#2} 
}}

\DeclareFontFamily{U}{MnSymbolC}{}
\DeclareSymbolFont{MnSyC}{U}{MnSymbolC}{m}{n}
\DeclareFontShape{U}{MnSymbolC}{m}{n}{
    <-6>  MnSymbolC5
   <6-7>  MnSymbolC6
   <7-8>  MnSymbolC7
   <8-9>  MnSymbolC8
   <9-10> MnSymbolC9
  <10-12> MnSymbolC10
  <12->   MnSymbolC12}{}
\DeclareMathSymbol{\intprod}{\mathbin}{MnSyC}{'270}

\newcommand{\calH}{\mathcal{H}}

\newcommand{\calG}{\mathcal{G}}

\newcommand{\calZ}{\mathcal{Z}}

\newcommand{\bR}{\mathbf{R}}
\newcommand{\bS}{\mathbf{S}}

\newcommand{\bN}{\mathbf{N}}

\newcommand{\bZ}{\mathbf{Z}}

\newcommand{\SO}{SO}

\newcommand{\so}{\mathfrak{so}}

\newcommand*\intdiff{\mathop{}\!\mathrm{d}}

\newcommand{\abs}[1]{\lvert#1\rvert}

\DeclareMathOperator{\VAR}{\mathbf{V}}

\newcommand{\Gr}{Gr}

\newcommand{\eps}{\epsilon}

\newcommand{\pd}{\partial}
\newcommand{\bdary}{\partial}

\newcommand{\clos}[1]{\overline{#1}}

\DeclareMathOperator{\indx}{index}
\DeclareMathOperator{\nullity}{nullity}

\DeclareMathOperator{\stab}{Stab}
\DeclareMathOperator{\orb}{Orb}

\DeclareMathOperator\Ric{Ric}
\DeclareMathOperator{\dist}{dist}

\theoremstyle{plain}
\newtheorem{thm}{Theorem}
\newtheorem*{thm*}{Theorem}
\newtheorem{lem}[thm]{Lemma}
\newtheorem*{lem*}{Lemma}
\newtheorem{prop}[thm]{Proposition}
\newtheorem{cor}[thm]{Corollary}
\newtheorem*{cor*}{Corollary}

\theoremstyle{definition}

\newtheorem*{ackn}{Acknowledgements}

\theoremstyle{remark}
\newtheorem{rem}[thm]{Remark}
\newtheorem*{rem*}{Remark}

\newtheorem*{claim*}{Claim}
\newtheorem*{notation*}{Notation}

\newtheorem*{quest*}{Question}

\theoremstyle{plain}

\title[Low index solutions of the Allen--Cahn equation on $\bS^3$]
{
Rigidity of low index solutions on $\bS^3$ \\
via a Frankel theorem \\
for the Allen--Cahn equation
}
\author{Fritz Hiesmayr}
\address{University College London, 25 Gordon Street, London WC1H 0AY}
\email{f.hiesmayr@ucl.ac.uk}

\begin{document}

\begin{abstract}
We prove a rigidity theorem in the style of Urbano for the Allen--Cahn equation
on the three-sphere: the critical points with Morse index five
are symmetric functions that vanish on a Clifford torus.
Moreover they realise the fifth width of the min-max spectrum
for the Allen--Cahn functional. 
We approach this problem by analysing the nullity and symmetries of these critical points.
We then prove a suitable Frankel-type theorem for their nodal sets,
generally valid in manifolds with positive Ricci curvature.
This plays a key role in establishing the conclusion, and further
allows us to derive ancillary rigidity results in spheres with larger dimension.
\end{abstract}

\maketitle

%

\section{Introduction}

F.\ Urbano gave a description of the minimal surfaces in the round $\bS^3$ with
low Morse index~\cite{UrbanoLowIndex}.
The only embedded surfaces with index at most five are
the equatorial spheres $S \subset \bS^3$, with index one, and the Clifford
tori $T \subset \bS^3$, with index five.
Our aim here is to prove an analogous result in the setting of the Allen--Cahn
equation.

Consequences of Urbano's theorem have found applications in numerous contexts,
most notably the proof of the Willmore conjecture by
Marques--Neves~\cite{MarquesNevesWillmore}.
Their proof relied on the development of the so-called Almgren--Pitts min-max
methods. 
In recent years an alternative min-max theory has grown, whose underlying
idea is to first find critical points of a certain semilinear elliptic functional.
(For the sake of this exposition we work in $\bS^3$, although much of the discussion
could remain unchanged in higher-dimensional spheres or general closed manifolds.)
This is the Allen--Cahn functional $E_\eps(u)
= \int_{\bS^3} \frac{\eps}{2} \abs{\nabla u}^2 + \eps^{-1} W(u)$,
which depends on a small parameter $\eps > 0$.
(For the sake of this exposition we may use the potential function
$W(t) = \frac{1}{4} (1 - t^2)^2$.)
This models the phase transition in a two-phase liquid. Its critical points
solve the elliptic Allen--Cahn equation
\begin{equation}
	\eps \Delta u - \eps^{-1} W'(u) = 0;
\end{equation}
we write $\calZ_\eps$ for the set of (classical) solutions of this equation.
As $\eps \to 0$ their values tend to congregate near $\pm 1$, separated
by a thin transition layer around a minimal surface.

A number of authors have made contributions to this alternative min-max theory;
we give a condensed summary of only those results most pertinent here.
Hutchinson--Tonegawa~\cite{HutchinsonTonegawa00} proved that given a sequence
of positive $\eps_j \to 0$ and critical points $u_j \in \calZ_{\eps_j}$, one
can form a sequence of varifolds $V_j$ which converge weakly to a stationary
integral varifold $V$.
When the Morse index of the $u_j$ is bounded along the sequence,
then the combined work of Tonegawa--Wickramasekera~\cite{TonegawaWickramasekera10}
and Guaraco~\cite{Guaraco15} establishes that
\begin{equation}
	V_j \to Q \abs{\Sigma} \text{ as $j \to \infty$},
\end{equation}
where $\Sigma \subset \bS^3$ is a smooth embedded minimal surface
and $Q \in \bZ_{>0}$ is its so-called multiplicity.
(Both of these rely fundamentally on the deep results
of Wickramasekera~\cite{Wickramasekera14} on the regularity of
stable codimension one stationary varifolds.)
Following this, Chodosh--Mantoulidis~\cite{ChodoshMantoulidis2018}
proved that the convergence is actually with multiplicity one, that is $Q = 1$.
(This is the main element not available in higher dimensions,
although it does hold in three-manifolds with positive Ricci curvature.)

It was Guaraco~\cite{Guaraco15} and Gaspar--Guaraco~\cite{GasparGuaraco2016} who
respectively implemented one-and multi-parameter min-max methods in the
Allen--Cahn setting. When working with $p \in \bZ_{>0}$ parameters, these
produce a critical point which realises the $p$-th \emph{width} $c_\eps(p)$
of the min-max spectrum of the Allen--Cahn functional.
With fixed $\eps > 0$ and varying $p$, these form a monotone increasing sequence,
\begin{equation}
0 <  c_\eps(1) \leq
	\cdots \leq c_\eps(p) \leq c_\eps(p+1) \leq \cdots.
\end{equation}

One is naturally led to draw comparisons with the sequence of widths
attained by Almgren--Pitts methods: $(\omega(p) \mid p \in \bZ_{>0})$.
Nurser~\cite{NurserMinMaxWidths} computed the first terms in this sequence
in the round $\bS^3$.
The first four have values $\omega(1) = \cdots = \omega(4) = 4 \pi$, and
are realised by equatorial spheres, and the next three widths are
$\omega(5) = \omega(6) = \omega(7) = 2 \pi^2$, and are realised by Clifford tori.
Recently Caju--Gaspar--Guaraco--Matthiesen~\cite{CajuGasparGuaracoMatthiesen}
proved an analogue for the first four min-max widths for the Allen--Cahn functional,
namely they are equal and are realised by symmetric functions
vanishing on equatorial spheres.
Moreover, they point out that for topological reasons the next width $c_\eps(5)$
must be distinct from the first four.
(The lower bounds for the phase transition spectrum of~\cite{GasparGuaraco2016} 
give an alternative justification for this.)

We prove that in fact $c_{\eps}(5)$ is realised by a function vanishing
on a Clifford torus $T$, with the same symmetries as $T$.
This result, stated in Corollary~\ref{cor_min_max_values} below, is a direct
consequence of the following rigidity theorem.

\begin{thm}
	\label{thm_main_result_clifford}
Given any $C > 1$, there is $\eps_0  > 0$ so that
for all $0 < \eps < \eps_0$ the following holds.
If a solution $u \in \calZ_\eps$ on $\bS^3$
has $\indx u \leq 5$
and $E_\eps(u) \leq C$, then its nodal set is either an equatorial sphere
or a Clifford torus, and $u$ is a symmetric solution around it.
\end{thm}

Its proof is given in the last section. We proceed by studying the nullity
and comparing it to the Killing nullity of $u$; this allows conclusions
about the symmetry of $u$.

\begin{cor}
\label{cor_min_max_values}
There is $\eps_0 > 0$ so that if $0 < \eps < \eps_0$
then any $u \in \calZ_\eps$ on $\bS^3$ with
$E_\eps(u) = c_\eps(5)$ is a symmetric solution vanishing on a Clifford torus.
\end{cor}

These are not the first rigidity results for the Allen--Cahn equation.
Besides the aforementioned~\cite{CajuGasparGuaracoMatthiesen} concerning the
ground states,
there is also recent work of
Guaraco--Marques--Neves~\cite{GuaracoMarquesNevesMultiplicityOne}
near non-degenerate minimal surfaces.
In~\cite{GuaracoMarquesNevesMultiplicityOne} the authors adapt the
curvature estimates of~\cite{ChodoshMantoulidis2018} to produce a non-trivial
Jacobi field on the limit minimal surface.
We sidestep the technical difficulties tied to such an approach,
and instead start by proving
a result inspired by \emph{Frankel's theorem}~\cite{FrankelPositiveCurvature,%
PetersenWilhelmFrankelsTheorem}.
%

Let $(M,g)$ be a closed manifold with positive Ricci curvature. Frankel's
theorem states that any two minimal hypersurfaces in $M$ must intersect.
This is not quite true for the nodal sets $Z(u_\eps^i) = \{ u_\eps^i = 0 \}$
of two solutions $u_\eps^i \in \calZ_\eps$ on $M$;
see for example~\cite[Expl.\ 1]{GuaracoMarquesNevesMultiplicityOne}. 
However, we prove that under mild topological hypotheses on the nodal sets
a Frankel-type result actually does hold.
The following theorem combines the statements of Corollary~\ref{cor_frankel_one_connected}
and Proposition~\ref{prop_frankel_both_connected}.
\begin{thm}
Let $(M,g)$ be a closed manifold of dimension $n+1\geq2$ with $\Ric > 0$.
Let $u_\eps^1,u_\eps^2 \neq \pm 1$ be two solutions of%
~\eqref{eq_allen_cahn} on $M$.
If 
\begin{enumerate}[label = (\Alph*),font = \upshape]
	\item either $Z(u_\eps^1)$ is separating and $Z(u_\eps^2)$ is connected,
	\item or $Z(u_\eps^1),Z(u_\eps^2)$ are both connected,
\end{enumerate}
then $Z(u_\eps^1) \cap Z(u_\eps^2) \neq \emptyset.$
\end{thm}

This Frankel property is of pivotal importance to our approach of the rigidity problem,
and is one of the main reasons for the brevity of our arguments. 
Moreover it allows us to obtain ancillary rigidity results in spheres with
dimension larger than three, the caveat being that in absence of%
~\cite{ChodoshMantoulidis2018} the convergence must be assumed to be with
multiplicity one.
We illustrate this with the Clifford-type minimal hypersurfaces 
\begin{equation}
	T_{p,q} =  \bS^p(\sqrt{p/n}) \times \bS^q(\sqrt{q/n})
	\subset \bS^{n+1},
\end{equation}
where $p,q \in \bZ_{>0}$ and $n = p + q \geq 2$.
By~\cite{AllardAlmgren_Uniqueness_Radial_Behaviour,%
HsiangLawson_Minimal_Submanifolds_of_Low_Cohomog} these are known
to have $\nu(T_{p,q}) = \nullity T_{p,q} = (p+1)(q+1)$.
Moreover they are homogeneous, and are stabilised by
$\SO(p) \times \SO(q) \subset \SO(n+2)$.
The rigidity result valid near these surfaces is similar to
Theorem~\ref{thm_main_result_clifford},
except that here no claim is made about the precise location of the nodal set.
(Here and in the remainder we write $(T_{p,q})_\delta \subset \bS^{n+1}$
for the tubular neighbourhood of $T_{p,q}$ with size $\delta > 0$.)

\begin{thm}
\label{thm_clifford_p_q}
Let $p,q \in \bZ_{>0}$ and $n = p+q \geq 2$.
There exist $0 < \eps_0,\delta < 1$ so that for all $0 < \eps < \eps_0$
there is a unique solution $u \in \calZ_{\eps}$ on $\bS^{n+1}$ with
$Z(u) \subset (T_{p,q})_\delta$ and
$(1 - \delta) \calH^n(T_{p,q}) \leq E_\eps(u) \leq (1 + \delta) \calH^n(T_{p,q})$,
up to rotation and change of sign.
Moreover $u$ is $\SO(p) \times \SO(q)$-symmetric, up to conjugation.
\end{thm}

In the more symmetric case where $p = q$ and $n = 2p$, the information about
the nodal set can be recovered. We state this corollary in a slightly different way.

\begin{cor}
	\label{cor_clifford_n}
Let $p \in \bZ_{>0}$ and $n = 2p$. Let $\eps_j \to 0$, and $u_j \in \calZ_{\eps_j}$
be so that $V(\eps_j,u_j) \to \abs{T_{p,p}}$. For large $j$, the nodal set of $u_j$
is a rigid copy of $T_{p,p}$ and $u_j$ is a symmetric solution around it.
\end{cor}

Theorem~\ref{thm_clifford_p_q} and Corollary~\ref{cor_clifford_n} are
proved in essentially the same way as the main result,
Theorem~\ref{thm_main_result_clifford}. We give the necessary modifications
to the higher-dimensional setting in Appendix~\ref{app_modifications} but
do not repeat the overlapping arguments.
In general, beyond the restrictive assumptions on the limit minimal surface%
---homogeneity and integrability through rotations---%
our arguments are quite flexible, and apply in closed manifolds with
positive Ricci curvature, of dimension three or larger.

\begin{ackn}
The author wishes to express his thanks to Costante Bellettini for
his support and invaluable discussions, as well as to Lucas Ambrozio,
Neshan Wickramasekera and Jason Lotay for helpful comments;
in fact this work was prompted by a question of Lucas Ambrozio.
The author's research is supported by the EPSRC under grant EP/S005641/1.
\end{ackn}

\section{Preliminary results}

Consider $\bS^3$ equipped with the standard round metric. For every $\eps > 0$
the \emph{$\eps$-Allen--Cahn equation} is the semilinear partial differential
equation 
\begin{equation}
	\label{eq_allen_cahn}
	\eps \Delta u - \eps^{-1} W'(u) = 0,
\end{equation}
where $W: \bR \to \bR$ is a regular non-negative even potential
with a positive, non-degenerate maximum at the origin
and two minima at $\pm 1$, where $W(\pm 1) = 0$.
(For example one may take $W(t) = (1 - t^2)^2/4$.)

This is the Euler--Lagrange equation of the \emph{Allen--Cahn functional},
$E_\eps(u) = \int_{\bS^3} \frac{\eps}{2} \abs{\nabla u}^2 + \eps^{-1} W(u)$.
We write $\calZ_\eps$ for the space of (classical) solutions of this equation.
There is a standard way to associate to each solution $u = u_\eps$
a \emph{varifold} $V_\eps = V(\eps,u)$, meaning a Radon measure on
the Grassmann bundle $\Gr_2(\bS^3)$.
Given any $\varphi \in C(\Gr_2(\bS^3))$, this is defined by
\begin{equation}
	\label{eq_varifold_definition}
	V_\eps(\varphi) =\frac{1}{2 \sigma}  \int_{\bS^3} \eps \abs{\nabla u(X)}^2
\varphi(X,T_X \{ u = u(X) \}) \intdiff \calH^3(X)
\end{equation}
where $\{ u = u(X) \}$ is the
level set of $u$ through $X$ and
$\sigma = \int_{-1}^1 \sqrt{W(t)/ 2} \intdiff t$ is a normalising constant.
(We use very little in the way of varifold theory in the remainder, though one may
consult~\cite{Simon84} for unfamiliar vocabulary.)

\subsection{Index and nullity}

Let $(\eps_j \mid j \in \bN)$ be a positive sequence with $\eps_j \to 0$ as $j \to \infty$,
and $(u_j \mid j \in \bN)$ be a corresponding sequence with
$u_j \in \calZ_{\eps_j}$ for all $j$.
We assume uniform bounds for the energy and that the Morse index of these
solutions is fixed, namely there are $C > 0$ and $k \in \bZ_{\geq 0}$
so that for all $j \in \bN$,
\begin{equation}
E_{\eps_j}(u_j) \leq C  \text{ and } \indx u_j = k.
\end{equation}
A subsequence of the $V_j = V(\eps_j,u_j)$ (which we extract without relabelling)
converges to a stationary integral varifold, as was shown by Hutchinson--Tonegawa%
~\cite{HutchinsonTonegawa00}.
Relying on the regularity theory developed by Wickramasekera~\cite{Wickramasekera14},
the combined work of Tonegawa--Wickramasekera~\cite{TonegawaWickramasekera10}
and Guaraco~\cite{Guaraco15} established that this limit is supported
in a smooth embedded minimal surface $\Sigma \subset \bS^3$, and
\begin{equation}
	V_j \to Q \abs{\Sigma}
\end{equation}
with integer multiplicity $Q \in \bZ_{>0}$. (Here and throughout we write
$\abs{\Sigma}$ for the unit density varifold associated to $\Sigma$.)
Using different methods, Hiesmayr~\cite{HiesmayrAllenCahn} and Gaspar~\cite{Gaspar2017}
both proved that
\begin{equation}
	\label{eq_index_bound}
	\indx \Sigma \leq k;
\end{equation}
the former building on previous work of Tonegawa~\cite{Tonegawa05} in the stable case,
the latter by adapting computations of Le~\cite{Le2011,Le2015}.

As the three-sphere $\bS^3$ has positive Ricci curvature, the results of
Chodosh--Mantoulidis~\cite{ChodoshMantoulidis2018}
give that the convergence is with multiplicity one,
meaning $Q = 1$ and $V_j \to \abs{\Sigma}$ as $j \to \infty$.
They also show that 
\begin{equation}
	\label{eq_lower_bound_augmented_index}
	\limsup_{j \to \infty} (\indx u_j + \nullity u_j)
	\leq \indx \Sigma + \nullity \Sigma.
\end{equation}
Combining this with the lower semicontinuity of the Morse index, we find that
for large enough $j$,
\begin{equation}
	\label{eq_nullity_lower_bound}
	\nullity u_j \leq \nullity \Sigma.
\end{equation}

We are especially interested in the situation where the Morse index of the
$u_j$ is low, specifically where it is five at the most.
In this range the classification of minimal surfaces of Urbano~\cite{UrbanoLowIndex}
shows that for the limit minimal surface $\Sigma$,
\begin{itemize}
	\item either $\indx \Sigma = 1$ and $\Sigma = S$ is an equatorial sphere,
	\item or $\indx \Sigma = 5$ and $\Sigma = T$ is a Clifford torus.
\end{itemize}
By~\cite{HsiangLawson_Minimal_Submanifolds_of_Low_Cohomog} they 
respectively have $\nullity S = 3$ and $\nullity T = 4$.

\subsection{Action by rotations}

The orthogonal group $\SO(4)$ acts on the space of minimal surfaces in $\bS^3$,
and preserves their index and nullity. 
Given a surface $\Sigma \subset \bS^3$, 
let as usual $\stab \Sigma  = \{ P \in \SO(4) \mid P( \Sigma) = \Sigma \}$ and
$\orb \Sigma = \{ P( \Sigma ) \subset \bS^3 \mid P \in \SO(4) \}$.
The former is a closed Lie subgroup of $\SO(4)$ denoted $G_\Sigma$, and 
$\dim \stab \Sigma + \dim \orb \Sigma = 6$, the dimension of $\SO(4)$.

Let $\Sigma$ be minimal and $L_\Sigma$ be the \emph{Jacobi operator} on $\Sigma$.
We write $\indx \Sigma \in \bZ_{\geq 0}$
for the number of strictly negative eigenvalues of $L_\Sigma$,
counted with multiplicity.
The kernel of $L_\Sigma$ is spanned by the so-called~\emph{Jacobi fields}.
Those that stem from the action of $\SO(4)$ are also called~\emph{Killing Jacobi fields}.
They span a linear subspace of $L^2(\Sigma)$ of dimension $\nu(\Sigma) = \dim \orb \Sigma$;
this is the \emph{Killing nullity}.
(Here we follow the conventions of Hsiang--Lawson%
~\cite{HsiangLawson_Minimal_Submanifolds_of_Low_Cohomog}.)
For any minimal surface $\Sigma \subset \bS^3$, 
\begin{equation}
\nu(\Sigma) \leq \nullity \Sigma.
\end{equation}
We are especially interested in surfaces for which $\nu$ equals the full nullity. 
Surfaces with this property are sometimes said to be~\emph{integrable through rotations}.
The equatorial sphere and the Clifford torus both have this property,
and quoting~\cite{HsiangLawson_Minimal_Submanifolds_of_Low_Cohomog} they respectively have
\begin{equation}
	\nullity S = \nu(S) = 3 \text{ and } \nullity T = \nu(T) = 4.
\end{equation}

Let $\Sigma \subset \bS^3$, with isotropy subgroup $G_\Sigma \subset \SO(4)$.
The surface is called \emph{homogeneous} if $G_\Sigma$ acts transitively on it.
This is equivalently expressed by saying that $\Sigma$ is an orbit
of the action of $G_\Sigma$ on $\bS^3$.
The equatorial sphere and the Clifford torus both satisfy this property;
in fact they are the only homogeneous minimal surfaces in $\bS^3$~%
\cite{HsiangLawson_Minimal_Submanifolds_of_Low_Cohomog}.

These notions have natural analogues in the setting of the Allen--Cahn equation. 
The group $\SO(4)$ gives a right action on functions defined on $\bS^3$
by pre-composition, where $P \in \SO(4)$ sends a function $u$ to $u \circ P$.
Given a function $u$ on $\bS^3$ we write
$\stab u = \{ P \in \SO(4) \mid u \circ P = u \}$ and
$\orb u = \{ u \circ P \mid P \in \SO(4) \}$.
Again we have $\dim \stab u + \dim \orb u = 6$.
(To see this fix any $\alpha \in (0,1)$; by elliptic regularity
any critical point has $u \in C^{2,\alpha}(\bS^3)$. Consider the orbit map
$P \in \SO(4) \mapsto u \circ P \in C^{2,\alpha}(\bS^3)$. Upon quotienting
by the stabiliser $G_u = \stab u$ this defines a homeomorphism onto its image.
In particular $\SO(4) / G_u$ and $\orb u$ have the same dimension, namely
$6 - \dim G_u$. Note that a rigorous rederivation of the analogous orbit-stabiliser
identity we quoted above for minimal surfaces $\Sigma \subset \bS^3$ would go
along the same lines, working in a suitable Banach space of embeddings into $\bS^3$.)

Let $M_{\eps}$ be the semilinear operator corresponding to~\eqref{eq_allen_cahn}.
The linearisation of $M_\eps$ around a function $u$ is (up to multiplication by $\eps$)
the linear operator $L_{\eps,u} = \Delta - \eps^{-2} W''(u)$.
(This describes the second variation of $-E_{\eps}$ at $u$ via integration by parts.)
This operator has finite-dimensional kernel, whose dimension is denoted
$\nullity u \in \bZ_{ \geq 0}$. We also write $\indx u \in \bZ_{\geq 0}$
for the number of strictly negative eigenvalues of $L_{\eps,u}$ counted with multiplicity.
The action of $\SO(4)$ preserves the Allen--Cahn functional,
$E_{\eps}( u \circ P) = E_\eps ( u)$ for all $P \in \SO(4)$.
It also maps $\calZ_\eps$ to itself, and for all $P \in \SO(4)$,
$\indx u \circ P = \indx u$ and $\nullity u \circ P = \nullity u$.
The invariance of $E_\eps$ under $\SO(4)$ also means that the action
generates functions in the kernel of $L_{\eps,u}$, which span a space of dimension
$\nu(u) \in \bZ_{\geq 0}$, the \emph{Killing nullity} of $u$.
As above we have $\nu(u) = \dim \orb u$ and 
\begin{equation}
	\label{eq_u_nullity_killing_comparison}
	\nu(u) \leq \nullity u.
\end{equation}
(The fact that $\nu(u) = \dim \orb u = 6 - \dim G_u$ is not hard to see. 
To derive this rigorously requires working again with the orbit map
$P \in \SO(4) \to u \circ P \in C^{2,\alpha}(\bS^3)$ we used above for
the orbit-stabiliser identity. This differentiates to the map $A \in \so(4)
\mapsto - \langle \nabla u, \xi_A \rangle$, where $\xi_A$ is the Killing vector
field corresponding to $A$. Writing $\mathfrak{g}_u \subset \so(4)$ for the 
Lie algebra corresponding to $G_u$, we obtain a linear isomorphism between
$\so(4) / \mathfrak{g}_u$ and the tangent space to the orbit of $u$.
In particular there are precisely $\nu = 6 - \dim \mathfrak{g}_u = 6 - \dim G_u$
Killing vector fields so that $\langle \nabla u , \xi_1 \rangle, \dots,
\langle \nabla u,\xi_\nu \rangle$ forms a linearly independent family.)

\subsection{Consequences of multiplicity one convergence}

Let $\eps_j \to 0$ and $(u_j \mid j \in \bN)$ be a sequence of critical points
in $\bS^3$ with $V(\eps_j,u_j) \to \abs{\Sigma}$, where $\Sigma \subset \bS^3$
is an embedded minimal surface.
Using either a simple calculation as in~\cite{CajuGasparAllenCahnWithSymmetry}
or Lemma~\ref{lem_stabiliser_conjugate} as justification, one finds that
for large $j$,
\begin{equation}
	\label{eq_killing_nullity_decreases_in_the_limit}
	\nu(\Sigma) \leq \nu(u_j).
\end{equation}
Assume additionally that $\Sigma$ is integrable through rotations, that is
$\nu(\Sigma) = \nullity \Sigma$. Stringing together~\eqref{eq_nullity_lower_bound},
\eqref{eq_u_nullity_killing_comparison}
and~\eqref{eq_killing_nullity_decreases_in_the_limit}, we find that
eventually
\begin{equation}
\nullity u_j = \nu(u_j).
\end{equation}

We now specialise this to the setting of low index in $\bS^3$, and show that
if $\indx u_j \leq 5$ then either $\indx u_j = 1$ or $5$, provided $j$
is large enough.
To this end assume that $E_{\eps_j}(u_j) \leq C$ and $\indx u_j \leq 4$
along the sequence. By~\eqref{eq_index_bound} and Urbano's classification,
$V(\eps_j,u_j)$ converges to an equatorial sphere $S \subset \bS^3$;
moreover by~\cite{ChodoshMantoulidis2018} the convergence is with multiplicity one,
that is $V(\eps_j,u_j) \to \abs{S}$ as $j \to \infty$.
From the above we have that for large $j$, $\nullity u_j = \nu(u_j)  = 3$.
At the same time $4 \geq \indx u_j + \nullity u_j \geq \indx u_j + 3$
by~\eqref{eq_lower_bound_augmented_index}.
As $u_j$ is unstable, we find $\indx u_j = 1$.

We apply this observation to the context of Corollary~\ref{cor_min_max_values},
where no assumption is made about the index of the solutions.
Instead there only imposes that the $u_j \in \calZ_{\eps_j}$
have $E_{\eps_j}(u_j) = c_{\eps_j}(5)$.
This notwithstanding, critical points obtained through five-parameter min-max
arguments as in~\cite{GasparGuaraco2016} have index at most five. Taking the
conclusions of Theorem~\ref{thm_main_result_clifford} for granted,
and as $c_{\eps_j}(4) < c_{\eps_j}(5)$, the functions obtained from min-max
methods must be symmetric solutions vanishing on a Clifford torus.
Moreover $c_{\eps_j}(5) \to 2 \pi^2$ as $j \to \infty$.
By~\cite{MarquesNevesWillmore}, the Clifford tori are the only minimal surfaces
in $\bS^3$ whose area has this value (or a fraction thereof),
so that $V(\eps_j,u_j) \to \abs{T}$, after extracting a subsequence
if necessary.
Arguing as above one obtains index bounds for the $u_j$, which prove that they
too must be of the rigid form described in Theorem~\ref{thm_main_result_clifford}.


\section{A Frankel-type result for the Allen--Cahn equation}

Let $(M,g)$ be a closed manifold with dimension $n+1$ and positive Ricci curvature.
We formulate two types of results for the Allen--Cahn equation in the style
of Frankel's theorem~\cite{FrankelPositiveCurvature}: 
the first concerning the solutions themselves, and the second their nodal sets.

\begin{prop}
\label{prop_frankel_allen_cahn}
Let $\eps > 0$ and $u_\eps^1 \neq u_\eps^2$
be two solutions of~\eqref{eq_allen_cahn} on $M$.
If $u_\eps^1 \leq u_\eps^2$ then one of the two is constant.
\end{prop}

We postpone the proof of this for now, and move on to the Frankel-type
result for nodal sets. 
Let us remark first that the nodal sets of two solutions
$u_\eps^1, u_\eps^2 \in \calZ_\eps$ do \emph{not} intersect in general.
Indeed, Guaraco--Marques--Neves~\cite[Expl.\ 1]{GuaracoMarquesNevesMultiplicityOne}
give a short construction of solutions near equatorial spheres in $\bS^{n+1}$
for which this fails.
(In their example one critical point vanishes on the equatorial sphere,
and the other vanishes along two hypersurfaces lying on either side of it,
a small distance away.)
One must therefore impose natural hypotheses on the nodal sets to ensure
that they meet.
Given a function $u$ on $M$, we write
$Z(u) = \{ u = 0 \}$ for its nodal set, and say that it is \emph{separating}
if $M \setminus Z(u)$ has exactly two connected components.

\begin{prop}
	\label{prop_frankel_nodal_set}
Let $\eps > 0$ and $u^1_\eps, u^2_\eps \neq \pm 1$
be two solutions of~\eqref{eq_allen_cahn} on $M$.
If $Z(u_\eps^1)$ is separating, then either $Z(u_\eps^1) \cap Z(u_\eps^2) \neq \emptyset$ or
\begin{equation}
	Z(u_\eps^2) \cap \{ u_\eps^1 > 0 \} \neq \emptyset 
	\text{ and }
	Z(u_\eps^2) \cap \{ u_\eps^1 < 0 \} \neq \emptyset.
\end{equation}
\end{prop}

\begin{proof}
We argue by contradiction, and assume that $Z(u_\eps^2) \subset \{ u_\eps^1 > 0 \}$.
Divide the complement of $Z(u_\eps^2)$ into its connected components,
$M \setminus Z(u_\eps^2) = \cup_{j = 0}^N U_j$ where $N \in \bZ_{\geq 0} \cup \{ \infty \}$.
(There are at most countably many of these, and the proof is the same whether
$N$ is finite or infinite.)
We decompose $\{ u_\eps^1 < 0 \}$ by conditioning on the $U_j$, and write
$\{ u_\eps^1 < 0 \} = \cup_{j=0}^N \{ u_\eps^1 < 0 \} \cap U_j$.
As $Z(u_\eps^1)$ is separating, only of these is non-empty, and
$\{ u_\eps^1 < 0 \} \subset U_0$ say.
Now on the one hand we may assume that $u_\eps^2 < 0$ on $U_0$, flipping
its sign if necessary. On the other hand, $\bdary U_0 \subset \{ u_\eps^1 > 0 \}$
and thus $\bdary (U_0 \cap \{ u_\eps^1 < 0 \}) \subset \bdary \{ u_\eps^1 < 0 \}$.
Therefore $u_\eps^1 = 0$ on $\bdary (U_0 \cap \{ u_\eps^1 < 0 \})$.
A simple consequence of the maximum principle%
---see~\cite[Cor.\ 7.4]{GuaracoMarquesNevesMultiplicityOne}---%
means that $u_\eps^2 < u_\eps^1$ on $U_0 \cap \{ u_\eps^1 \leq 0 \}$.
This actually holds on the whole $U_0$, as $u_\eps^2 < 0 \leq u_\eps^1$
on $U_0 \setminus \{ u_\eps^1 \leq 0 \}$.
Let $U_j$ be one of the remaining components. 
If $u_\eps^2 \leq 0$ in $U_j$ then $u_\eps^2 \leq u_\eps^1$, so we may assume
$u_\eps^2 \geq 0$.
Then $u_\eps^2 = 0$ on $\bdary U_j$ while $u_\eps^1 > 0$ on $\clos{U}_j$.
Arguing as above we find $u_\eps^2 < u_\eps^1$ on $\clos{U}_j$.
As $U_j$ was arbitrary, we conclude that $u_\eps^2 < u_\eps^1$ on $M$,
which contradicts Proposition~\ref{prop_frankel_allen_cahn} because neither
function is constant equal $\pm 1$.
\end{proof}
\begin{rem}
Technically~\cite[Cor.\ 7.4]{GuaracoMarquesNevesMultiplicityOne} is stated
on regular domains, but an inspection of the proof reveals this hypothesis
to not be necessary.
\end{rem}

The following is an immediate consequence.

\begin{cor}
\label{cor_frankel_one_connected}
If $Z(u^1_\eps)$ is separating and $Z(u_\eps^2) \neq \emptyset$ is connected
then $Z(u^1_\eps) \cap Z(u^2_\eps) \neq \emptyset.$
\end{cor}

The next proposition can be derived independently,
using a similar argument. Either Corollary~\ref{cor_frankel_one_connected}
or Proposition~\ref{prop_frankel_both_connected} would be sufficient for
our purposes; we include both for the sake of completeness.

\begin{prop}
\label{prop_frankel_both_connected}
If $Z(u^1_\eps),Z(u^2_\eps) \neq \emptyset$ are connected then 
$Z(u^1_\eps) \cap Z(u_\eps^2) \neq \emptyset$.
\end{prop}
\begin{proof}
Again we argue by contradiction, assuming that $Z(u_\eps^1) \cap Z(u_\eps^2) = \emptyset$.
Split $Z = Z(u_\eps^1) = Z_- \cup Z_+$ according to the sign of $u_\eps^2$.
As $Z$ is connected, only one of these can be non-empty, and $Z = Z_-$ say.
After perhaps flipping the sign of $u_\eps^1$, we find that $u_\eps^2 < u_\eps^1$
in $\{ u_\eps^1 \leq 0 \}$ by~\cite[Cor.\ 7.4]{GuaracoMarquesNevesMultiplicityOne}.
The mirror argument shows that also $u_\eps^2 < u_\eps^1$ in $\{u_\eps^2 \geq 0 \}$.
On the remaining region $u_\eps^2 < 0 < u_\eps^1$, which means 
$u_\eps^2 < u_\eps^1$ is established on the entirety of $M$; the conclusion
follows from Proposition~\ref{prop_frankel_allen_cahn}.
\end{proof}

We now turn to the proof of Proposition~\ref{prop_frankel_allen_cahn}.
For this we use the \emph{parabolic Allen--Cahn equation}
with initial data a bounded function $u_0 \in C^1(M)$,
which given $\eps > 0$ and $T \in (0,\infty]$ is defined to be
\begin{equation}
	\label{eq_parabolic_allen_cahn}
\begin{cases}
	u_t = \eps \Delta u - \eps^{-1} W'(u) & \text{in $M \times [0,T)$}, \\
	u(0,\cdot) = u_0 & \text{on $M$}.
\end{cases}
\end{equation}

\begin{lem}
\label{lem_flow_subsolution}
Let $u_0 \in C^1(M)$ be a weak subsolution of~\eqref{eq_allen_cahn}
with $\abs{u_0} \leq 1$.
Then the solution to~\eqref{eq_parabolic_allen_cahn} exists for all time,
and $-1 \leq u(s,\cdot) \leq u(t,\cdot) \leq 1$ for all $0 \leq s \leq t$.
As $t \to \infty$, $u(t,\cdot)$ converges to a smooth function $u_+$.
Moreover $u_+$ is a stable solution of~\eqref{eq_allen_cahn}, and can be
characterised as the least element of $\{ v \in \calZ_\eps \mid v \geq u_0 \}$.
\end{lem}

\begin{proof}
Using the classical theory of parabolic PDE, we find that the solution
$u$ of~\eqref{eq_parabolic_allen_cahn} is unique, smooth and exists for all time.
To obtain the monotonicity of the flow, note that $u_t$ satisfies the following
parabolic equation: $\pd_t u_t = \eps \Delta u_t - \eps^{-1} W''(u) u_t$.
As initially $u_t(0,\cdot) \geq 0$, the parabolic maximum principle forces 
$u_t \geq 0$ for all $t \geq 0$. Moreover, this is strict unless $u_t = 0$,
which happens precisely if $u_0$ is a solution of~\eqref{eq_allen_cahn}.
Together with classical parabolic Schauder estimates, this monotonicity guarantees the
convergence of $u(t,\cdot)$ as $t \to \infty$, with limit the smooth function $u_+$.
Moreover $u_+$ is a classical solution of~\eqref{eq_allen_cahn}.
As the set $\{ v \in \calZ_\eps \mid v \geq u_0 \}$ works as a barrier for the flow,
the function $u_+$ must be its least element.

It remains to see the stability of $u_+$; there are various ways to obtain this.
For example, assume that $u_+$ is unstable and let $\varphi_1 > 0$ be its first
eigenfunction, with eigenvalue $\lambda_1 < 0$.
A quick computation reveals that for a suitably small $\theta \in (0,1)$, 
$u_+ - \theta \varphi_1 < u_+$ is a supersolution of~\eqref{eq_allen_cahn}.
Indeed
$\eps \Delta(u_+ - \theta \varphi_1) - \eps^{-1} W'(u_+ - \theta \varphi_1)
= \varphi_1 (\theta \lambda_1
- \eps^{-1} \int_0^\theta W^{'''}(u_+ - t \varphi_1) t \varphi_1 \intdiff t)$.
%
%
If we were to apply the arguments above, we would find that solving the
parabolic Allen--Cahn equation with initial datum $u_+ - \theta \varphi_1$
yields a strictly decreasing solution. 
Perhaps after adjusting $\theta$ to a smaller value so that
$u_+ - \theta \varphi_1 > u_0$, this acts as an upper barrier
and makes $u(t,\cdot) \to u_+$ absurd.
\end{proof}

We use this to prove Proposition~\ref{prop_frankel_allen_cahn}.

\begin{proof}
In a manifold with positive Ricci curvature, the only
stable solutions of~\eqref{eq_allen_cahn} are the constant functions with
values one of $ \{ -1,0,1 \}$.
This is because, as a quick computation shows,
$\delta^2 E_\eps(u)(\abs{\nabla u},\abs{\nabla u})
= - \eps \int_M \abs{\nabla^2 u}^2
- \abs{\nabla \abs{\nabla u}}^2 + \Ric(\nabla u, \nabla u)$.
This is non-negative precisely when $\abs{\nabla u} \equiv 0$ and $u$ is constant.

Let $u_1 \leq u_2$ be two distinct solutions of~\eqref{eq_allen_cahn},
and assume that $u_1$ is not constant equal $\pm 1$.
By the maximum principle  $u_1 < u_2$ on $M$.
(If $u_1 = 0$ then this would imply that $u_2$ is constant equal one.)
Let $\varphi_1 > 0$ be the first eigenfunction of the linearised operator $L_{\eps,u_1}$,
with eigenvalue $\lambda_1 < 0$. The same computation as in the proof of
Lemma~\ref{lem_flow_subsolution} shows that for sufficiently small $\theta \in (0,1)$,
$u_1 + \theta \varphi_1$ is a subsolution of~\eqref{eq_allen_cahn}.
Take $\theta > 0$ small enough that still $u_1 + \theta \varphi_1 < u_2$. 
Then solving~\eqref{eq_parabolic_allen_cahn} with this function as an initial
datum we obtain a strictly increasing family $u(t,\cdot)$ of functions with
$u_1 < u(t,\cdot) \leq u_2$. 
By Lemma~\ref{lem_flow_subsolution} this converges to a stable solution $u_+$
when we let $t \to \infty$. The characterisation of $u_+$ given there shows
that $u_+ \leq u_2$, while the fact that $\Ric > 0$ means that $u_+$ is constant
equal $1$. 
\end{proof}

\section{Symmetry of solutions}

\subsection{Stabilisers and weak convergence}
The Lie group $\SO(4)$ can be endowed with a bi-invariant metric, which induces
a distance function $d$.
Write $\calG$ for the set of closed subgroups of $\SO(4)$.
When endowed with the topology induced by the Hausdorff distance,
the couple $(\calG,d_H)$ forms a compact metric space.
Moreover, we have the following useful result.
(We state this for an arbitrary compact Lie group $G$, although for our
purposes $G = \SO(4)$ or $\SO(n+2)$ would be sufficient.
Moreover given a closed subgroup $H \subset G$, we let
$(H)_\tau = \{ P \in G \mid \dist(P,H) < \tau \}$ be 
its open tubular neighbourhood of size $\tau > 0$.)
\begin{lem}[\cite{MontgomeryZippinLieGroups}]
\label{lem_subgroups_conjugate}
Let $G$ be a compact Lie group, and $H$ be a closed subgroup of $G$.
There is $\tau > 0$ so that every subgroup $H' \in \calG$
with $H' \subset (H)_\tau$ is conjugate to a subgroup of $H$.
\end{lem}

Let $\eps_j \to 0$ be a sequence of positive scalars and $(u_j \mid j \in \bN)$
be a sequence of solutions of~\eqref{eq_allen_cahn}, with respectively 
$u_j \in \calZ_{\eps_j}$.
We assume that they additionally have uniformly bounded Allen--Cahn energy and
$\indx u_{\eps_j} = 5$ for all $j$, whence $V(\eps_j,u_j) \to \abs{T}$ as
$j \to \infty$.
We abbreviate 
\begin{equation}
	G_j = \stab u_j \text{ and } G_T = \stab T. 
\end{equation}
Moreover, we may assume throughout that $j$ is large enough that%
~\eqref{eq_nullity_lower_bound} holds, ensuring that $\nullity u_j \leq 4$.
We use Lemma~\ref{lem_subgroups_conjugate} to show that eventually
the stabilisers $G_j$ and $G_T$ are conjugate subgroups of $\SO(4)$.

First we point that out that $\SO(4)$ defines a natural action on the space of 
two-varifolds $\VAR_2(\bS^3)$. A rotation $P \in \SO(4)$ maps
$V \in \VAR_2(\bS^3)$ to its push-forward $P_{\#} V$.
(Here we implicitly identify $P \in \SO(4)$ with the isometry it induces on $\bS^3$,
as we have been doing.)
This action is continuous in the varifold topology.
To see this, let $(P,V) \in \SO(4) \times \VAR_2(\bS^3)$ be arbitrary,
and consider two sequences $P_k \to P$ and $V_k \to V$,
the latter being in the varifold topology.
Given $\varphi \in C(\Gr_2(\bS^3))$, one has
$(P_{k\#} V_k - P_{\#} V) \varphi =
P_{k\#} (V_k - V) \varphi + (P_{k\#} V - P_{\#} V) \varphi$, both of which
tend to zero as $k \to \infty$.
We define the stabiliser and orbit of $V$ in the usual way.
For an embedded surface $\Sigma \subset \bS^3$ and all $P \in \SO(4)$,
it holds that $P_{\#} \abs{\Sigma} = \abs{P(\Sigma)}$ and we need not distinguish
between the stabiliser and orbit of $\Sigma$ as an embedded surface or as a varifold.

Let $u \in \calZ_\eps$, and the varifold $V(\eps,u)$ be defined
as in~\eqref{eq_varifold_definition}. Via a simple change of variable we find
\begin{equation}
	\label{eq_orthogonal_action_veps}
	P_{\#} V(\eps,u) = V(\eps,u \circ P^{-1})
	\text{ for all $P \in \SO(4)$.}
\end{equation}
To see this, let $X \in \bS^3$ be arbitrary and $Y = P(X)$.
Assume without loss of generality that
$\abs{\nabla u}(X) = \abs{\nabla (u \circ P^{-1})}(Y) > 0$,
ensuring that the level sets $\{ u = u(X) \}$
and $\{ u \circ P^{-1} = (u \circ P^{-1})(Y) \}$ 
are regular near $X$ and $Y$ respectively.
The map $P$ induces on $\Gr_2(\bS^3)$ sends $(X,T_X \{ u = u(X) \})$
to $(Y,T_Y \{ u \circ P^{-1} = (u \circ P^{-1})(Y) \}$.
Given any $\varphi \in C(\Gr_2(\bS^3))$, we may thus compute
$[P_{\#} V(\eps,u)] (\varphi) = V(\eps,u)(\varphi \circ P)$ to be
equal $V(\eps,u \circ P^{-1})(\varphi)$:
\begin{multline}
	\frac{1}{2\sigma}\int_{\bS^{3}} \eps \abs{\nabla u}^2(X)
	(\varphi \circ P)(X,T_X \{ u = u(X) \}) \intdiff \calH^3(X) \\
	= \frac{1}{2 \sigma} \int_{\bS^3} \eps \abs{\nabla (u \circ P^{-1})}^2(Y)
	\varphi(Y,T_Y \{ u \circ P^{-1} = (u \circ P^{-1})(Y) \} \intdiff \calH^3(Y).
\end{multline}

Although the compared actions are on the right and left respectively,%
it follows from~\eqref{eq_orthogonal_action_veps} that $\stab u \subset \stab V(\eps,u)$,
and specialised to our sequence $G_j \subset \stab V(\eps_j,u_j)$ for all $j$.

\begin{lem}
\label{lem_stabiliser_conjugate}
Let $\eps_j \to 0$ and $u_j \in \calZ_{\eps_j}$ be as described above,
with $V(\eps_j,u_j) \to \abs{T}$ as $j \to \infty$.
For large $j$, $G_j$ is conjugate to a subgroup of
$G_T = \SO(2) \times \SO(2)$.
\end{lem}
\begin{proof}
Consider a sequence $(P_j \mid j \in \bN)$ with $P_j \in \stab u_j$.
Upon extracting a subsequence we may assume that it converges to some $P \in \SO(4)$.
As $V(\eps_j,u_j) \to \abs{T}$ we get on the one hand
$P_{j\#} V(\eps_j,u_j) \to P_{\#} \abs{T}$.
On the other hand $P_{j\#} V(\eps_j,u_j) = V(\eps_j,u_j \circ P_j^{-1}) = V(\eps_j,u_j)$
by construction, so $P_{\#} \abs{T} = \abs{T}$ and $P \in G_T = \stab T$.
Via another extraction argument, we find that given any $\tau > 0$
there is $J(\tau) \in \bN$ so that $G_j \subset (G_T)_\tau$ when $j \geq J(\tau)$.
By Lemma~\ref{lem_subgroups_conjugate}, $G_j$ is conjugate to a subgroup of $G_T$.
\end{proof}

Therefore eventually $\dim \stab u_j \leq 2$ and $\nu(u_j) \geq 4$.
Combining this with~\eqref{eq_nullity_lower_bound}, we find that for large $j$
\begin{equation}
	\label{eq_nullity_killing_nullity_eventually_coincide}
	\nu(u_j) = \nullity u_j = 4.
\end{equation}
Thus $G_j$ is conjugate to $G_T$; further given any $\tau > 0$ there is $J(\tau) \in \bN$
so that  for $j \geq J(\tau)$, there is $P_j \in \SO(4)$ with $d(P_j,I) < \tau$ and
$P_j^{-1} G_j P_j = G_T$.

\subsection{Proof of Theorem~\ref{thm_main_result_clifford}}

Applying the results of~\cite{BrezisOswaldSublinearPDE} in the present setting,
one obtains the following; see also~\cite{CajuGasparGuaracoMatthiesen} for
an alternative construction of the symmetric critical point.
\begin{lem}
\label{lem_unique_symmetric_solution}
There is $\eps_0 = \eps_0(T) > 0$ so that for all $0 < \eps < \eps_0$
there is a unique function $u_{T,\eps} = u \in \calZ_\eps$ with
$\{ u = 0 \} = T$, up to change of sign. Moreover $u_{T,\eps}$ is invariant under $G_T$.
\end{lem}
\begin{proof}
It is convenient to write $\bS^3 = \{ (X,Y) \in \bR^2 \times \bR^2 \mid
\abs{X}^2 +\abs{Y}^2 = 1 \}$. The complement of $T$ has two connected components $U_{\pm}$,
and we write $U_+ = \{ (X,Y) \in \bS^3 \mid \abs{X} < \abs{Y} \}$
and $U_{-} = \{ (X,Y) \in \bS^3 \mid \abs{X} > \abs{Y} \}$.
These two regions are isometric via the involution $(X,Y) \in \bS^3 \mapsto (Y,X)$.
Let $U = U_{\pm}$ and $\lambda_1(\Delta;U)$ be the first eigenvalue
of the Laplacian on $U$ with Dirichlet eigenvalues.
By~\cite{BrezisOswaldSublinearPDE}, provided
\begin{equation}
	0 < \eps <  \{ W''(0)/ \lambda_1(\Delta;U) \}^{1/2},
\end{equation}
there is a unique solution $u_{\pm}$ on $U_{\pm}$ respectively of the system
\begin{equation}
\begin{cases}
	\eps \Delta u - \eps^{-1} W'(u) = 0  &\text{in $U$}, \\
	u > 0 &\text{in $U$}, \\
	u = 0 &\text{on $\bdary U$}.
\end{cases}
\end{equation}
As this system is invariant under the action of $G_T$, the solutions $u_{\pm}$
must further be $G_T$-invariant.
Moreover they are symmetric under the involution above, that is $u_+(X,Y) = u_{-}(Y,X)$
for all $(X,Y) \in \bS^3$ with $\abs{X} < \abs{Y}$.
By elliptic regularity $u_{\pm}$ are respectively smooth up to and including
the boundary of $U_{\pm}$.
We define the function $u$ by setting
\begin{equation}
\label{eq_defn_fn_u}
u(X,Y) 	= 
\begin{cases}
	u_+(X,Y) &\text{if $\abs{X} \leq \abs{Y}$,} \\
	-u_-(X,Y) &\text{if $\abs{X} > \abs{Y}$}.
\end{cases}
\end{equation}
By the symmetry of $u_\pm$ under the involution, this function is smooth
and solves~\eqref{eq_allen_cahn} on $\bS^3$.
Thus there exists at least one solution of~\eqref{eq_allen_cahn} which vanishes
precisely on $T$. 
Conversely, if an arbitrary solution $v \in \calZ_\eps$ had $Z(v_{\eps}) = T$ then
its restrictions to ${U_{\pm}}$ would both have a sign, and thus $v$ is
forced to coincide with $u$ up to change of sign.
\end{proof}

We conclude with a proof of the main result: Theorem~\ref{thm_main_result_clifford}.

\begin{proof}
Let $u_\eps \in \calZ_\eps$ be a solution with $E_\eps(u_\eps) \leq C$ and 
$\indx u_\eps = 5$. Let a small $\tau > 0$ be given, and $\eps> 0$ be small 
enough in terms of $\tau,D$ that~\eqref{eq_nullity_killing_nullity_eventually_coincide}
holds and there is $P_{\eps} \in \SO(4)$ with $d(P_\eps,I) < \tau$
and $P_{\eps}^{-1} G_\eps P_\eps = G_T$, where $G_\eps = \stab u_\eps$.
As the nodal set $Z(u_\eps) = \{ u_\eps = 0 \}$ converges to $T$ with
respect to Hausdorff distance we may moreover assume that $Z(u_\eps) \subset (T)_\tau$.
Let $v_\eps = u_\eps \circ P_\eps \in \calZ_\eps$; the energy, index and nullities
are unaffected by this operation and $v_\eps$ is stabilised by $G_T$.
Therefore its nodal set is union of orbits of $G_T$.
Appealing to~\cite{ChodoshMantoulidis2018} for example, it must be connected and
so be of the form $Z(v_\eps) = \{ \dist(\cdot,T) = \delta \}$
for some $\delta$ which tends to zero as $\tau,\eps \to 0$.
Let $u_{T,\eps}$ be the symmetric solution around the Clifford torus from
Lemma~\ref{lem_unique_symmetric_solution}. The Frankel-type property of
Corollary~\ref{cor_frankel_one_connected} forces $\delta = 0$ and
$Z(v_\eps) = Z(u_{T,\eps})$.
(In fact here both nodal sets are connected and separating.)
By Lemma~\ref{lem_unique_symmetric_solution} up to a change of sign $v_\eps = u_{T,\eps}$,
which concludes the proof.
\end{proof}

\appendix

\section{Modifications in higher dimensions}

\label{app_modifications}

The only result that needs to be slightly altered in higher dimensions,
when working with the hypersurface $T_{p,q} \subset \bS^{n+1}$,
is Lemma~\ref{lem_unique_symmetric_solution}.
Even then, in case $p = q$ the statement and its proof remain valid with no
changes. However the less symmetric case where $p \neq q$ calls for a more
complicated construction; Caju--Gaspar~\cite{CajuGasparAllenCahnWithSymmetry}
prove the following. (Their result is valid more broadly;
we give a modified statement specialised to the present context.)
\begin{lem}[{\cite[Thm.\ 1.1]{CajuGasparAllenCahnWithSymmetry}}]
Let $p,q > 0$ and $n = p + q \geq 3$. Given any $\delta > 0$ there is $\eps_0 > 0$
so that for all $0 < \eps < \eps_0$ there is a solution $u_{p,q,\eps} \in \calZ_{\eps}$
on $\bS^{n+1}$ with $Z(u_{p,q,\eps}) \subset (T_{p,q})_\delta$ and
$(1-\delta) \calH^n(T_{p,q}) \leq E_\eps(u_{p,q,\eps}) \leq (1 + \delta)\calH^n(T_{p,q})$.
\end{lem}
Applying~\cite{BrezisOswaldSublinearPDE} here gives that any
other solution $v_\eps$ of~\eqref{eq_allen_cahn} on $\bS^{n+1}$ with
$Z(v_{\eps}) = Z(u_{p,q,\eps})$ must coincide with $u_{p,q,\eps}$ up to 
a possible change of sign.
There are two ways to justify this here. The first is via the results of~%
~\cite{ChodoshMantoulidis2018}, which show that for $\eps > 0$ small enough
the nodal set of the $u_{p,q,\eps}$ converge smoothly to the limit surface $T_{p,q}$.
The uniqueness is then a direct consequence of~\cite{BrezisOswaldSublinearPDE},
applied in the two regions making up $\bS^{n+1} \setminus Z(u_{p,q,\eps})$.
For an argument that does not rely open the convergence of the nodal sets,
one may combine~\cite{HardtSimon89} with~\cite{BrezisOswaldSublinearPDE}
to obtain the following general lemma.
\begin{lem}
\label{lem_nodal_set_allen_cahn}
Let $(M,g)$ be closed. Let $\eps > 0$ and $u_\eps^1,u_\eps^2$ be two solutions
of~\eqref{eq_allen_cahn} on $M$. If $Z(u_{\eps}^1) = Z(u_{\eps}^2)$,
then $u_\eps^1 = \pm u_\eps^2$.
\end{lem}
\begin{proof}
Write $Z = Z(u_\eps^1) = Z(u_\eps^2)$. The cases where $Z = \emptyset$ or
$Z = M$ are trivial, and we leave them aside. Divide the complement of $Z$ into
its connected components, say $M \setminus Z = \cup_{j = 0}^N U_j$, where
$N \in \bZ_{\geq 0} \cup \{ \infty \}$.
The two functions $u_\eps^1,u_\eps^2$ have a sign in the interior of each $U_j$,
and vanish on $\bdary U_j$. By~\cite{BrezisOswaldSublinearPDE} they are equal
up to a change of  sign; however this needs to be chosen consistently across all
regions of $M \setminus Z$.
Call two regions $U_j,U_k$ \emph{adjacent} if $\calH^{n-1}(\bdary U_j \cap \bdary U_j)
\neq 0$. By~\cite{HardtSimon89} the nodal set may be decomposed like
$Z = R \cup S$, where $R$ is the set of regular points, that is those $X \in Z$
so that for some $\rho  > 0$, $B_\rho(X) \cap Z$ is a $C^1$ $(n-1)$-dimensional
submanifold, and $S$ is a countably $(n-2)$--rectifiable set.
If $U_j,U_k$ are adjacent then $\bdary U_j \cap \bdary U_k$ contains a regular point
$X$ say, and there is $\rho > 0$ so that $\bdary U_j \cap \bdary U_k \cap B_\rho(X)
= R \cap B_\rho(X)$.
The regularity of the boundary near $X$ allows the application of%
~\cite[Lem.\ 1]{BrezisOswaldSublinearPDE} to deduce that
$\frac{\partial u_1}{\partial \nu}, \frac{\partial u_2}{\partial \nu} \neq 0$
on $\bdary U_j \cap \bdary U_k \cap B_\rho(X)$.
It follows that the respective signs of $u_1,u_2$ on $U_j$ determines their 
signs on $U_k$, and vice-versa.
Now let $U_j,U_k$ be two connected components, which are not necessarily adjacent.
There is a path $\gamma: [0,1] \to M$ with endpoints
$\gamma(0) \in U_j$ and $\gamma(1) \in U_k$. 
Using a perturbation analogous to that used in~\cite[Lem.\ A.1]{SimonWickramasekera16}
one may arrange for $\gamma([0,1]) \cap S = \emptyset$. The curve $\gamma$ runs
through finitely many regions of $M \setminus Z$. List them as $U_{j_1} = U_j,
U_{j_2},\dots,U_{j_D} = U_k$, which are pairwise adjacent in this order.
Therefore the sign of $u_1,u_2$ on $U_j$ determines their sign on $U_k$ and
vice-versa; this concludes the proof.
\end{proof}

For the remaining steps in the proof of Theorem~\ref{thm_clifford_p_q} one may thus
follow the arguments we used for Theorem~\ref{thm_main_result_clifford}.

\bibliographystyle{alpha}

\bibliography{min_surfaces_refs.bib}

\end{document}